\documentclass[12pt]{amsart}

\usepackage{amsmath}
\usepackage{amssymb}
\usepackage{amsfonts}
\usepackage{amsopn}
\usepackage{amsthm}
\usepackage[all]{xy}

\swapnumbers

\usepackage{amsthm}

\newtheorem{thm}[equation]{Theorem}
\newtheorem{lem}[equation]{Lemma}
\newtheorem{cor}[equation]{Corollary}
\newtheorem{prop}[equation]{Proposition}

\newtheorem*{thm*}{Theorem}
\newtheorem*{prop*}{Proposition}
\newtheorem*{cor*}{Corollary}
\newtheorem*{lem*}{Lemma}
\newtheorem*{MT*}{Main Theorem}

\theoremstyle{definition}
\newtheorem{defn}[equation]{Definition}
\newtheorem*{defn*}{Definition}

\newtheorem{eg}[equation]{Example}

\theoremstyle{remark}

\newtheorem*{rmk*}{Remark}
\newtheorem*{rmks*}{Remarks}













\makeatletter
\newenvironment{borel*}
{\smallskip \refstepcounter{equation}\noindent{\textbf{\theequation.}}}
{\global\@ignoretrue}
\makeatother

\newcommand{\flist}[1]{\hangindent\leftmargini\textup{(1)}\hskip\labelsep {#1}%
\begin{enumerate}
\setcounter{enumi}{1}
}

\newcommand{\zz}{\mathbb{Z}} 
\newcommand{\rr}{\mathbb{R}} 
\newcommand{\cc}{\mathbb{C}} 
 
\newcommand{\qq}{\mathbb{Q}} 

\newcommand{\A}{\mathrm{A}}

\newcommand{\D}{\mathrm{D}}
\newcommand{\Kill}{\mathrm{Kill}}
\newcommand{\E}{\mathrm{E}}
\newcommand{\F}{\mathrm{F}_4}
\newcommand{\G}{\mathrm{G}_2}

\newcommand{\Alt}{\mathrm{Alt}}

\DeclareMathOperator{\Aut}{\mathrm{Aut}}

\DeclareMathOperator{\PGL}{\mathrm{PGL}}
\DeclareMathOperator{\SL}{\mathrm{SL}}

\DeclareMathOperator{\GL}{\mathrm{GL}} 
\DeclareMathOperator{\PSL}{\mathrm{PSL}}

\DeclareMathOperator{\SO}{SO}
\DeclareMathOperator{\PSO}{PSO}
\DeclareMathOperator{\hspin}{HSpin}
\DeclareMathOperator{\Spin}{Spin}

\newcommand{\stbtmat}[4]{\left( \begin{smallmatrix} #1&#2 \\ #3&#4 \end{smallmatrix} \right) }

\newcommand{\ra}{\rightarrow}
\newcommand{\Z}{\mathbb{Z}}
\newcommand{\R}{\mathbb{R}}
\newcommand{\Q}{\mathbb{Q}}
\newcommand{\la}{\lambda}
\newcommand{\Zm}[1]{\Z/{#1}\Z}

\newcommand{\kx}{k^\times}

\newcommand{\qform}[1]{{\left\langle{#1}\right\rangle}}

\numberwithin{equation}{section}

\title{Degree 5 invariant of $E_8$}
\author{Skip Garibaldi}
\address{(Garibaldi) Department of Mathematics \& Computer Science, Emory University, Atlanta, GA 30322, USA}
\email{skip@member.ams.org}
\urladdr{http://www.mathcs.emory.edu/{\textasciitilde}skip/}

\author{Nikita Semenov}
\address{(Semenov) Mathematisches Institut der LMU M\"unchen, Theresienstr.~39,
80333~M\"unchen, Germany}
\email{semenov@mathematik.uni-muenchen.de}

\date{{\tt Version of \today.}}

\begin{document}
\maketitle

\section{Introduction}

Throughout this paper, $k$ denotes a field of characteristic 0.  We write $E_8$ for the split simple algebraic group with Killing-Cartan type $E_8$.  The Galois cohomology set $H^1(k, E_8)$ classifies simple algebraic groups of type $E_8$ over $k$.  One of the goals of the theory of algebraic groups over arbitrary fields is to understand the groups of type $E_8$, equivalently, understand the set $H^1(k, E_8)$.

The main tool is the \emph{Rost invariant} 
\[
r_{E_8} \!: H^1(*, E_8) \ra H^1(*, \mu_{60}^{\otimes 2})
\]
discovered by Markus Rost and explained in Merkurjev's portion of the book \cite{GMS}.  It is a morphism of functors from the category of fields over $k$ to the category of pointed sets.  We put
\[
H^1(k, E_8)_0 := \{ \eta \in H^1(k, E_8) \mid \text{$r_{E_8}(\eta)$ has odd order} \}.
\]
In \cite[Corollary 8.7]{S08}, the second author proved that there is a morphism of functors:
\[
u \!: H^1(*, E_8)_0 \ra H^5(*, \Zm2).
\]
This is the degree 5 invariant from the title.

Let now $G$ be a group of type $E_8$.  It corresponds with a canonical element of $H^1(k, E_8)$, so it makes sense to speak of ``the Rost invariant of $G$"; we denote it by $r(G)$.  Suppose now that $r(G)$ has odd order, so $G$ belongs to $H^1(k, E_8)_0$.  The second author also proved in \cite{S08}:
\begin{equation} \label{odd.split}
\parbox{4in}{\emph{$u(G) = 0$ if and only if there is an odd-degree extension of $k$ that splits $G$.}}
\end{equation}
For example, the compact group $G$ of type $\E_8$ over $\R$ has Rost invariant zero and $u(G) = (-1)^5$.

As an obvious corollary, we have: 
\begin{equation}
\parbox{4in}{\emph{If $k$ has cohomological dimension $\le 2$, then every $k$-group of type $E_8$ is split by an odd-degree extension of $k$.}}
\end{equation}
Serre's ``Conjecture II" for groups of type $E_8$ is that in fact every group of type $E_8$ over such a field is split.

The purpose of this note is to calculate the value of $u(G)$ for certain $G \in H^1(k, E_8)_0$ and to give some applications of $u$.

%%%%%%%%%%%%%%%%%%%%%%%%%%%%%%%%%%%%%%%%%%%%%
\section{Tits's construction of groups of type $E_8$}

\begin{borel*} \label{Tits.const}
There are inclusions of algebraic groups $G_2 \times F_4 \subset E_8$, where $G_2$ and $F_4$ denote split groups of those types.  Furthermore, this embedding is essentially unique.  Applying Galois cohomology gives a function $H^1(k, G_2) \times H^1(k, F_4) \ra H^1(k, E_8)$.  The first two sets classify octonion $k$-algebras and Albert $k$-algebras respectively, so this map gives a construction by Galois descent of groups of type $E_8$:
\[
\fbox{octonion $k$-algebras} \times \fbox{Albert $k$-algebras} \ra \fbox{groups of type $E_8$}
\]
Jacques Tits gave concrete formulas on the level of Lie algebras for this construction in \cite{Ti:const}, see also \cite{Jac:ex}.  This method of constructing groups of type $E_8$ is known as the \emph{Tits construction}.  (Really, Tits's construction is more general and gives other kinds of groups as well.  The variety of possibilities is summarized in Freudenthal's magic square as in \cite[p.~540]{Inv}.  However, the flavor in all cases is the same, and this case is the most interesting.)

Our purpose is to compute the value of $u$ on those groups of type $E_8$ with Rost invariant of odd order (so that it makes sense to speak of $u$) and arising from Tits's construction.  We do this in Proposition \ref{pro}.
\end{borel*}

\begin{borel*} \label{FG.inv}
Following \cite{Inv},
we write $f_3(-)$ for the even component of the Rost invariant of an Albert algebra or an octonion algebra (equivalently, a group of type $\F$ or $\G$).  We write $g_3(-)$ for the odd component of the Rost invariant of an Albert algebra; such algebras also have an invariant $f_5$ taking values in $H^5(k, \Zm2)$.  \emph{An Albert algebra $A$ has $g_3(A) = 0$ and $f_5(A) = 0$ iff $A$ has a nonzero nilpotent, iff the group $\Aut(A)$ is isotropic.}
\end{borel*}

\smallskip

Suppose now that $G \in H^1(k, E_8)$ is the image of an octonion algebra $O$ and an Albert algebra $A$.  It follows from a twisting argument as in the proof of Lemma 5.8 in \cite{GQ} that 
\[
r(G) = r_{G_2}(O) + r_{F_4}(A).
\]
In particular, $G$ belongs to $H^1(k, E_8)_0$ if and only if 
$f_3(O) + f_3(A) = 0$ in $H^3(k, \Zm2)$, i.e., if and only if $f_3(O) = f_3(A)$.

\begin{defn}
Define
\[
t \!: H^1(*, F_4) \ra H^1(*, E_8)_0
\]
by sending an Albert $k$-algebra $A$ to the group of type $E_8$ constructed from $A$ and the octonion algebra with norm form $f_3(A)$, via Tits's construction from \ref{Tits.const}.  By the preceding paragraph, $r(G) = g_3(A) \in H^3(k, \Zm3)$, so $G$ does indeed belong to $H^1(k, E_8)_0$.
\end{defn}

\begin{eg} \label{Ti.nil}
\emph{If $A$ has a (nonzero) nilpotent element, then the group $t(A)$ is split.}  Indeed, $g_3(A)$ is zero so $t(A)$ is in the kernel of the Rost invariant.  Also, $t(A)$ is isotropic because it contains the isotropic subgroup $\Aut(A)$, hence $t(A)$ is split by, e.g., \cite[Prop.~12.1(1)]{G:deg16}.
\end{eg}

\begin{eg} \label{Ti.R}
In case $k = \Q$, there are exactly three Albert algebras up to isomorphism.  All have $g_3 = 0$; they are distinguished by the values of $f_3$ and $f_5$.  :
\begin{center}
\begin{tabular}{c|c||c}
$f_3(A)$ & $f_5(A)$ & $t(A)$\\
\hline
$0$ & $0$ & split by Example \ref{Ti.nil}\\
$(-1)^3$ & $0$ & split by Example \ref{Ti.nil}\\
$(-1)^3$ & $(-1)^5$ & anisotropic by \cite[p.~118]{Jac:ex}
\end{tabular}
\end{center}
It follows from Chernousov's Hasse Principle for groups of type $\E_8$ \cite{PlatRap} that for every number field $K$ with a unique real place, the set $H^1(K, \E_8)_0$ has two elements: the split group and the anisotropic group constructed as in the last line of the table.
\end{eg}

\begin{prop}\label{pro}
For every Albert $k$-algebra $A$, we have:
\[
u(t(A)) = f_5(A) \quad \in H^5(k, \Zm2).
\]
\end{prop}

\begin{proof}
The composition $ut$ is an invariant $H^1(*, \F) \ra H^5(*, \Zm2)$, hence is given by
\[
ut(A) = \la_5 + \la_2 \cdot f_3(A) + \la_0 \cdot f_5(A)
\]
for uniquely determined elements $\la_i \in H^i(\Q, \Zm2)$, see \cite[p.~50]{GMS}.  

We apply this formula to each of the three lines in the table from Example \ref{Ti.R}.
Obviously $u$ of the split $\E_8$ is zero, so the first line gives:
\[
0 = u(\text{split $\E_8$}) = \la_5   \quad \in H^5(\Q, \Zm2).
\]
Applying this to the second line gives:
\[
0 = u(\text{split $\E_8$}) = \la_2 \cdot (-1)^3 \in H^5(\Q, \Zm2).
\]
For the last line, 
$u$ of the compact $\E_8$ is $(-1)^5$ by \eqref{odd.split}, see the end of \cite{S08} for details.  We find:
\[
(-1)^5 = u(\text{compact $\E_8$}) = \la_0 \cdot (-1)^5,
\]
so $\la_0$ equals 1 in $H^0(\Q, \Zm2) = \Zm2$.

To show that $\lambda_2=0$ we proceed as follows. Consider the pure transcendental
extension $F=\qq(x,y,z,a,b)$ and let $H$ be the group of type $\F$ with
$f_3(H)=(x,y,z)$, $f_5(H)=f_3(H)\cdot(a,b)$ and $g_3(H)=0$.
Then $ut(H)=f_5(H)+f_3(H)\cdot\lambda_2$.

Let $K$ be a generic splitting field for the symbol $f_5(H)$.
Since $H_K$ is isotropic, the resulting group $t(H)$ of type $\E_8$ is isotropic
over $K$, and, since it has trivial Rost invariant, it splits over $K$ \cite[Prop.~12.1]{G:deg16}.  Obviously, $ut(H)$ is killed by $K$.
Therefore $f_3(H)\cdot\lambda_2$ is zero over $K$.  If $f_3(H) \cdot \lambda_2$ is zero over $F$, then by taking residues we see that $\la_2$ is zero in $H^2(\Q(a,b), \Zm2)$, hence also in $H^2(\Q, \Zm2)$.  Otherwise, $f_3(H) \cdot \la_2$ is equal to $f_5(H)$ by \cite[Theorem~2.1]{OViVo}, and again completing and taking residues with respect to the $x$-, $y$-, and $z$-adic valuations, we find that $\la_2 = (a, b) \in H^2(\Q(a, b), \Zm2)$.  But this is impossible because $\la_2$ is defined over $\Q$.  This proves that $\la_2 = 0$.
\end{proof}

\begin{cor}
Let $G$ be as in Proposition~\ref{pro}. Let
$\Kill_-$ denote the Killing form of $-$ and $\E_8$ the split group.
Then $$\langle 60\rangle(\Kill_G-\Kill_{\E_8})=2^3\cdot u(G)\in I^8(k).$$
\end{cor}
\begin{proof}
Follows from \cite[13.5 and Example 15.9]{G:deg16} and Proposition~\ref{pro}.
\end{proof}

\begin{eg} \label{aniso.eg}
Whatever field $k$ of characteristic zero one starts with, there is an extension $K/k$ that supports an anisotropic 5-Pfister quadratic form $q_5$---one can adjoin 5 indeterminates to $k$, for example.  Let $q_3$ be a 3-Pfister form dividing $q_5$ and let $A$ be the Albert $K$-algebra with $f_d(A) = e_d(q_d)$ for $d = 3, 5$.  The group $G := t(A)$ of type $\E_8$ over $K$ has Rost invariant zero yet $u(G) = f_5(A)$ nonzero by Proposition \ref{pro}.  In particular, $G$ is not split, hence is anisotropic by \cite[Prop.~12.1(1)]{G:deg16}.
\end{eg}

Example 15.9 in \cite{G:deg16} produced anisotropic groups of type $\E_8$ in a similar manner, but used the Killing form to see that the resulting groups were anisotropic; that method does not work if $-1$ is a square in $k$.  Roughly, Example \ref{aniso.eg} above exhibits more anisotropic groups because $u$ is a finer invariant than the Killing form.

%%%%%%%%%%%%%%%%%%%%%%%%%%%%%%%%%%%%%%%%%%%%%%
\section{Invariants of $H^1(*, \Spin_{16})_0$}

Recall from \cite[pp.~436, 437]{Inv} that the Rost invariant of a class $\eta \in H^1(k, \Spin_{16})$ is given by the formula
\[
r_{\Spin_{16}}(\eta) = e_3(q_\eta) \quad \in H^3(k, \Zm2)
\]
where $q_\eta$ is the 16-dimensional quadratic form in $I^3k$ corresponding to the image of $\eta$ in $H^1(k, \SO_{16})$ and $e_3$ is the Arason invariant.  it follows that $\eta$ belongs to the kernel of the Rost invariant if and only if $q_\eta$ belongs to $I^4 k$.

We can quickly find some invariants of the kernel $H^1(k, \Spin_{16})_0$ of the Rost invariant.  For $\eta$ in that set, $q_\eta$ is $\qform{\alpha_\eta} \gamma$ for some $\alpha_\eta \in \kx$ and some 4-Pfister quadratic form $\gamma$ \cite[X.5.6]{Lam}.  (One can take $\alpha_\eta$ to be any element of $\kx$ represented by $q_\eta$ \cite[X.1.8]{Lam}.)  We define invariants $f_d \!: H^1(*, \Spin_{16})_0 \ra H^d(*, \Zm2)$ for $d = 4, 5$ via:
\[
f_4(\eta) := e_4(q_\eta) \quad \text{and} \quad f_5(\eta) := (\alpha_\eta) \cdot e_4(q_\eta),
\]
where $e_4$ is the usual additive map $I^4(*) \ra H^4(*, \Zm2)$.  If $q_\eta$ is isotropic, then $q_\eta$ is hyperbolic and $e_4(q_\eta)$ is zero, so the value of $f_5(\eta)$ depends only on $\eta$ and not on the choice of $\alpha_\eta$, see \cite[10.2]{G:lens}.

We can identify two more (candidates for) invariants of $H^1(*, \Spin_{16})_0$.   The split $E_8$ has a subgroup isomorphic to $\hspin_{16}$, the nontrivial quotient of $\Spin_{16}$ that is neither adjoint (i.e., not $\PSO_{16}$) nor $\SO_{16}$.  Further, the composition
\[
H^1(*, \Spin_{16}) \ra H^1(*, \hspin_{16}) \ra H^1(*, \E_8) \xrightarrow{r_{E_8}} H^3(*, \Zm{60}(2))
\]
is the Rost invariant of $\Spin_{16}$ \cite[(5.2)]{G:deg16}.  We find a morphism of functors
\[
H^1(*, \Spin_{16})_0 \ra H^1(*, \E_8)_0.
\]
Composing this with the invariant $u$ gives an invariant
\[
u_5 \!: H^1(*, \Spin_{16})_0 \ra H^5(*, \Zm2).
\]
(Roughly speaking, we have used the invariant $u$ of $H^1(*, \E_8)_0$ to get an invariant of $H^1(*, \Spin_{16})_0$ in the same way that Rost used the $f_5$ invariant of $H^1(*, \F)$ to get an invariant of $H^1(*, \Spin_9)$, see \cite{Rost:spin14} or \cite[18.9]{G:lens}.)

The purpose of this section is to prove:

\begin{prop} \label{s16.basis}
The invariants $H^1(*, \Spin_{16})_0 \ra H^\bullet(*, \Zm2)$ form a free $H^\bullet(k, \Zm2)$-module with basis
\[
1, f_4, f_5, u_5, u_6,
\]
where the invariant $u_6$ is given by the formula $u_6(\eta) := (\alpha_\eta) \cdot u_5(\eta)$.
\end{prop}

We first replace $\Spin_{16}$ with a more tractable group.  The first author described in \cite[\S11]{G:deg16} a subgroup of $\hspin_{16}$ isomorphic to $\PGL_2^{\times 4}$.  Examining the root system data for this subgroup given in Tables 7B and 11 of that paper, we see that the inverse image of this subgroup in $\Spin_{16}$ is a subgroup $H$ obtained by modding $\SL_2^{\times 4}$ out by the subgroup generated by $(-1, -1, 1, 1)$, $(-1, 1, -1, 1)$, and $(-1, 1, 1, -1)$.  Indeed, in the notation of that paper, each copy of $\PGL_2^{\times 4}$ lifts to a copy of $\SL_2$ with a maximal torus the image of one of the four elements of the root lattice of $D_8$:
\begin{gather*}
\delta_1 + 2 \delta_2 + 3 \delta_3 + 4 \delta_4 + 3 \delta_5 + 2 \delta_6 + \delta_7,  \quad
\delta_1 - \delta_3 + \delta_5 - \delta_7, \\
\delta_1 + \delta_3 - \delta_5 - \delta_7, \quad \text{or} \quad
\delta_1 + 2\delta_2 + \delta_3 - \delta_5 - 2\delta_6 - \delta_7.
\end{gather*}
(Here $\delta_i$ denotes the simple root of $\D_8$ that is denoted $\alpha_i$ in \cite{Bourbaki}.)  The center of each copy of $\SL_2$ has nontrivial element
\[
h_{\delta_1}(-1) \, h_{\delta_3}(-1) \, h_{\delta_5}(-1) \, h_{\delta_7}(-1)
\]
in the notation of \cite{Steinberg}.  This defines a homomorphism $\mu_2 \ra H$ that gives a short exact sequence:
\[
1 \ra \mu_2 \ra H \ra \PGL_2^{\times 4} \ra 1.
\]
The image of $H^1(k, H)$ in $H^1(k, \PGL_2)^{\times 4}$ consists of quadruples $(Q_1, Q_2, Q_3, Q_4)$ of quaternion algebras so that $Q_1 \otimes Q_2 \otimes Q_3 \otimes Q_4$ is split.

Let $\varphi$ map the Klein four-group $V := \Zm2 \times \Zm2$ into $(\SL_2 \times \SL_2)/\mu_2$ via
\[
\varphi(1, 0) := \left( \stbtmat{-1}{0}{0}{1}, \stbtmat{-1}{0}{0}{1} \right) \quad \text{and} \quad
\varphi(0, 1) := \left( \stbtmat{0}{1}{-1}{0}, \stbtmat{0}{1}{-1}{0} \right).
\]
This defines a homomorphism.  Twisting $(\SL_2 \times \SL_2)/\mu_2$ by a pair $(a, b) \in \kx / k^{\times 2} \times \kx / k^{\times 2} = H^1(k, V)$ gives $(\SL(Q) \times \SL(Q))/\mu_2$, where $Q$ denotes the quaternion algebra $(a, b)$.  (Of course, composing $\varphi$ with either of the projections $(\SL_2 \times \SL_2)/\mu_2 \ra \PGL_2$ sends $(a, b)$ to the same quaternion algebra $Q$.)  The composition
\[
V \times V \xrightarrow{\varphi \times \varphi} (\SL_2 \times \SL_2)/\mu_2 \times (\SL_2 \times \SL_2)/\mu_2 \ra H
\]
gives a map whose image does not meet the center of $\Spin_{16}$, which we denote by $Z$.  This gives a homorphism
$Z \times V \times V \ra \Spin_{16}$.

\begin{lem} \label{s16.surj}
For every extension $K/k$,
the map $Z \times V \times V \ra \Spin_{16}$ defined above induces a map $H^1(K, Z \times V \times V) \ra H^1(K, \Spin_{16})_0$ that is a surjection.
\end{lem}

\begin{proof}
Fix $\nu \in H^1(K, Z \times V \times V)$; we write $\zeta$ for its image in $H^1(K, Z)$ and $Q_1, Q_2$ respectively for its images under the two projections $H^1(K, Z \times V \times V) \ra H^1(K, V) \ra H^1(K, \PGL_2)$.  It follows from the description of the subgroup $\PGL_2^{\times 4}$ in $\hspin_{16}$ and the map $H^1(K, \hspin_{16}) \ra H^1(K, \PSO_{16})$ in \cite[\S4]{G:deg16} that the image of $\nu$ in $H^1(K, \PSO_{16})$ is the same as the image of the quadratic form $q_1 \otimes q_2$ under $H^1(K, \SO_{16}) \ra H^1(K, \PSO_{16})$, where $q_i$ denotes the norm form of $Q_i$.  As the Rost invariant of $\Spin_{16}$ ``factors through" $\SO_{16}$ to give the Arason invariant, the image of $\nu$ in $H^1(K, \Spin_{16})$ is in the kernel of the Rost invariant.

As for surjectivity, if $\eta$ is in $H^1(K, \Spin_{16})_0$ then $q_\eta = \qform{\alpha} q_1 q_2$ where $q_1, q_2$ are 2-Pfister forms.  We define $\nu$ to be $(0, Q_1, Q_2)$ where $Q_i$ is a quaternion algebra with norm $q_i$.  Then the image of $\nu$ in $H^1(K, \Spin_{16})$ is $\zeta \cdot \eta$ for some $\zeta \in H^1(K, Z)$, and we deduce that $\zeta \cdot \nu$ maps to $\eta$.
\end{proof}

It follows from the lemma (using \cite[5.3]{G:lens}) that the module of invariants $H^1(*, \Spin_{16})_0 \ra H^\bullet(*, \Zm2)$ injects into the module of invariants $H^1(*, Z \times V \times V) \ra H^\bullet(*, \Zm2)$.  But we know this larger module by \cite[p.~40]{GMS} or \cite[6.7]{G:lens}: it is spanned by products $\pi_1 \pi_2 \pi_3$ for $\pi_1 \!: H^1(*, Z) \ra H^\bullet(*, \Zm2)$ and $\pi_2, \pi_3$ invariants $H^1(*, V) \ra H^\bullet(*, \Zm2)$ composed with projection on the 2nd or 3rd term in the product.

For the rest of this section, we use the notation $(\zeta, Q_1, Q_2)$ for elements of $H^1(K, Z \times V \times V)$ as in the proof of Lemma \ref{s16.surj}.

\begin{lem} \label{s16.zero}
Fix $(\zeta, Q_1, Q_2) \in H^1(K, Z \times V \times V)$.  If $[Q_1] \cdot [Q_2] = 0$ in $H^4(K, \Zm2)$, then the image of $(\zeta, Q_1, Q_2)$ in $H^1(K, \Spin_{16})$ is zero.
\end{lem}

\begin{proof}
The image $\eta$ of $(\zeta, Q_1, Q_2)$ has $q_\eta = \qform{\alpha_\eta} q_1 q_2$ where $q_i$ is the norm of $Q_i$.  If $[Q_1] \cdot [Q_2]$ is zero, then $q_1 \otimes q_2$ is in $I^5(k)$ and so is hyperbolic.  It follows that the image of $\eta$ in $H^1(K, \PSO_{16})$ is zero, hence $\eta$ is in the image of $H^1(K, Z)$.  But $\Spin_{16}$ is split semisimple, so the image of $H^1$ of the center in $H^1$ of the group is zero.
\end{proof}

\begin{proof}[Proof of Proposition \ref{s16.basis}] The invariant $f_4$ sends $(\zeta, Q_1, Q_2)$ to $[Q_1] \cdot [Q_2]$.  Lemma \ref{s16.zero} combined with arguments like those in \cite[pp.~43, 44]{GMS} shows that every invariant of $H^1(*, \Spin_{16})_0$ restricts to one of the form $\la + \phi \cdot f_4$ for uniquely determined $\la \in H^\bullet(k, \Zm2)$ and $\phi \!: H^1(*, Z) \ra H^\bullet(*, \Zm2)$, i.e., is given by the formula
\[
(\zeta, Q_1, Q_2) \mapsto \la + \phi(\zeta) \cdot [Q_1] \cdot [Q_2].
\]
The collection of such invariants of $H^1(*, Z \times V \times V)$ forms a free $H^\bullet(k, \Zm2)$-module with basis
\[
1,\quad f_4,\quad \chi_v \cdot f_4, \quad \chi_h \cdot f_4,\quad \chi_v \cdot \chi_h \cdot f_4,
\]
where $\chi_v$ and $\chi_h$ denote the maps $H^1(*, Z) \ra H^1(*, \Zm2)$ defined by restricting to $Z$ the vector representation $\Spin_{16} \ra \SO_{16}$ and the half-spin representation $\Spin_{16} \ra \hspin_{16}$ implicit in our root-system description above.   Obviously, $f_5$ restricts to be $\chi_v \cdot f_4$.

At this point, it suffices to prove:
\begin{equation} \label{u5.res}
\text{\emph{$u_5$ restricts to be $\chi_h \cdot f_4$.}}
\end{equation}
Indeed, this statement implies that $u_5$ is zero when $q_\eta$ is isotropic, hence by \cite[10.2]{G:lens} the formula for $u_6$ gives a well-defined invariant; it obviously restricts to $\chi_v \cdot u_5 = \chi_v \cdot \chi_h \cdot f_4$ on $H^1(*, Z \times V \times V)$.  Spanning and linear independence follow from the previous paragraph.

We now prove \eqref{u5.res}.  The restriction of $u_5$ sends zero to zero, so it is $\phi \cdot f_4$ for some $\phi \!: H^1(*, Z) \ra H^1(*, \Zm2)$ that itself sends zero to zero.  Therefore (by \cite[p.~40]{GMS}) $\phi$ is induced by some homomorphism $\chi \!: Z \ra \Zm2$.  As $u_5$ is defined by pulling back along the map $\Spin_{16} \ra \E_8$, one quickly sees that $\chi$ must be zero or $\chi_h$.  As $\chi$, the zero invariant, and $\chi_h$ are all defined over $\Q$, it suffices to prove that $\chi$ is not the zero invariant in the case where $k = \Q$.  Example 15.1 of \cite{G:deg16} gives a class $\nu \in H^1(\R, Z \times V \times V)$ whose image in $H^1(\R, \E_8)$ is the compact $\E_8$, on which $u$ is nonzero.  Hence the restriction of $u_5$ to $Z \times V \times V$ is not zero over $\Q$ and must be $\chi_h \cdot f_4$.
\end{proof}

%%%%%%%%%%%%%%%%%%%%%%%%%%%%%%%%%%%%%%%%%%%%%%
\section{Essential dimension of $H^1(*, \Spin_{16})_0$}

The following is a corollary of the previous section:

\begin{cor}
The essential dimension of the functor $H^1(*, \Spin_{16})_0$ over every field of characteristic zero is $6$.
\end{cor}

\begin{proof}
The existence of the nonzero invariant $u_6 \!: H^1(*, \Spin_{16})_0 \ra H^6(*, \Zm2)$ implies that the essential dimension is at least 6 by \cite[Lemma 6.9]{RY}; this is the interesting inequality.  One can deduce that the essential dimension is at most 6 by, 
for example, the surjectivity in Lemma \ref{s16.surj}.
\end{proof}

By way of contrast, the essential dimension of the functor $H^1(*, \Spin_{16})$ (without restricting to the kernel of the Rost invariant) is 24 by \cite[Remark 3.9]{BRV}.

%%%%%%%%%%%%%%%%%%%%%%%%%%%%%%%%%%%%%%%%%%%%%
\section{Galois descent for representations of finite groups}

In this section, we restate some obsevations of Serre from \cite{Se00}
and \cite{GrRy98} regarding projective embeddings of simple groups
in exceptional algebraic groups. Combining these results with the $u$-invariant
for $\E_8$ gives some new embeddings results, see Example~\ref{E8.eg} below.

Let $A$ be an abstract finite group and $G$ a split semisimple linear algebraic
group defined over $\qq$. Fix a faithful representation $G\to\GL_N$ defined over $\qq$.
Further on fix a monomorphism $\varphi\colon A\to G(\overline k)$, where $\overline k$ is an algebraically
closed field of characteristic zero.

\begin{defn}
Let $F$ be a field extension of $\qq$.
We say that the character of a representation $A\to G(\overline k)\to\GL_N(\overline k)$
is {\it defined over $F$}, if all its values belong to $F$.
\end{defn}

Let $\chi$ be the character of the representation
$$\varphi\colon A\to G(\overline k)\to\GL_N(\overline k)$$ and $F$ its
field of definition. Assume additionally that
$Z_{G(\overline k)}(A)=1$ and that there is exactly one
$G(\overline k)$-conjugacy class of
homomorphisms $A\to G(\overline k)$ with character $\chi$.

The following theorem can be extracted from Serre's paper \cite[2.5.3]{Se00}:
\begin{thm}\label{thm2}
In the above notation there exists a twisted form $G_0$ of $G$ defined over $F$ together
with a monomorphism $A\to G_0(F)$.
Moreover, for a field extension $K/F$
there is a representation $A\to G(K)$ with character $\chi$
iff $G\simeq G_0$ over $K$.
\end{thm}
\begin{proof}
Let $$P=\{\alpha\colon A\to G\mid \alpha\text{ is a representation with character }\chi\}.$$
Then $G$ acts on $P$ by conjugation. By assumptions on $A$ and $G$ this action is
transitive. Moreover, the condition on
the centralizer guaranties that
this action is simply transitive, i.e., for any $\alpha,\beta\in P(\overline K)$
there exists a unique $g\in G(\overline K)$ with $\beta=\alpha^g$.
Thus, $P$ is a $G$-torsor. Since all values of the character $\chi$
belong to the field $F$, the torsor $P$ is defined over $F$.

Let $\eta\in H^1(F,G)$ be the $1$-cocycle corresponding to the torsor $P$.
Then $\sigma\cdot\varphi=\eta_\sigma^{-1}\varphi\eta_\sigma$ for all $\sigma$
in the absolute Galois group $\mathrm{Gal}(\overline F/F)$.
Define now $G_0$ as the twisted form of $G$ over $F$ by the torsor $P$.
The group $G_0$ is defined out of $G(\overline F)$ by a twisted Galois action:
$$\sigma*g=\eta_\sigma(\sigma\cdot g)\eta_\sigma^{-1}.$$
Now it is easy to see that the homomorphism $\varphi\colon A\to G(\overline F)$
is an $F$-defined homomorphism $A\to G_0(F)$.

Let $K/F$ be a field extension. If there is a reperesentation $A\to G(K)$
with character $\chi$, then obviously $G$ and $G_0$ are isomorphic over $K$.
Conversely, if $G$ and $G_0$ are isomorphic over $K$, then the image of the
cocycle $\eta$ in $H^1(K,\Aut(G))$ is zero. Since the centralizer of $A$
in $G$ is trivial, the group $G$ is adjoint. Therefore $\eta$ is already zero
in $H^1(K,G)$.
\end{proof}

To characterize the isomorphism criterion of Theorem~\ref{thm2}
we need the following proposition.

\begin{prop}\label{pro2}
For each Killing-Cartan type $\Phi$ in the table
\begin{center}
\begin{tabular}{c|ccc}
Type $\Phi$ & $\F$ & $\G$ & $\E_8$ \\ \hline
$n$ & $3$ & $3$ & $5$
\end{tabular}
\end{center}
there is a unique algebraic group $G_0$ of type $\Phi$ that is compact at every real place of every (equivalently, a particular) number field; it is defined over $\Q$.  For every field $K$ of characteristic zero and $n$ as in the table, the following are equivalent:
\begin{enumerate}
\item $G_0\otimes K$ is split.
\item $(-1)^n=0\in H^n(K,\zz/2)$.
\item $-1$ is a sum of $2^{n-1}$ squares of the field $K$.
\end{enumerate}
\end{prop}
\begin{proof}
The first sentence is a standard part of the Kneser-Harder-Chernousov Hasse
principle. The group $G_0$ is split at every finite place.

For the second claim, all cases but $\E_8$ are well-known.  For $\E_8$, if $G_0 \otimes K$ is split,
then $(-1)^5$ is zero by the existence of $u$; see \ref{odd.split}.
For the converse, $G_0$ equals $t(A)$
where $A$ is the unique Albert $\qq$-algebra with no nilpotents
(see Example~\ref{Ti.R}).  If $(-1)^5$ --- i.e., $f_5(A)$ --- is zero in
$H^5(K, \Zm2)$, then $A \otimes K$ has nilpotents and $G_0 \otimes K$
is split by Example \ref{Ti.nil}.
\end{proof}

In the following examples we denote as $\Alt_l$ the alternating group of degree $l$
and as $\zeta_l=e^{2\pi i/l}$ a primitive $l$-th root of unity.

\begin{eg}[type $\G$] \label{G2.eg}
Let $G$ denote the split group of type $\G$, $A=G(\mathbb{F}_2)$ (resp. $\PSL(2,8)$, $\PSL(2,13)$),
and $K$ a field of characteristic zero. Then there is
an embedding $A\to G(K)$ iff $-1$ is a sum of $4$ squares of $K$
and $\zeta_9+\bar\zeta_9\in K$ (for $\PSL(2,8)$), resp. $\sqrt{13}\in K$ (for $\PSL(2,13)$).

Indeed, fix the minimal fundamental representation $G\to\GL_7$.
By \cite[Theorem~9(3,4,5)]{A87} there is a representation $\varphi\colon A\to G(\overline k)$
whose character $\chi$ is defined over $F=\qq$
(resp. $F=\qq(\zeta_9+\bar\zeta_9)$, $F=\qq(\sqrt{13})$).
Moreover, $G$ acts transitively on the homomorphisms $A\to G(\overline k)$
with character $\chi$ (see \cite{A87} and \cite[Cor.~1 and 2]{Griess}).

By \cite[9.3(1)]{A87} the representation $\varphi$ is irreducible.
Therefore $Z_{G(\overline k)}(A)=1$.
Thus, all conditions of Theorem~\ref{thm2}
are satisfied. Therefore there is a twisted form $G_0$ of $G$ defined over $F$
and an embedding $A\to G_0$.

In particular, there is an embedding $A\to G_0(\rr)$.
Since any finite subgroup of a Lie group is contained in its maximal
compact subgroup,
it is easy to see that $G_0\otimes_F\rr$ is compact for all embeddings
of $F$ into $\rr$. Moreover, by Theorem~\ref{thm2} we have an embedding
$A\to G(K)$ iff $G_0$ and $G$ are isomorphic over $K$. By Proposition~\ref{pro2}
the latter occurs iff $-1$ is a sum of $4$ squares of $K$.

(Thus, we have recapitulated the argument from \cite[2.5.3]{Se00}).
\end{eg}

\begin{eg}[type $\E_8$] \label{E8.eg}
Let $G$ denote the split group of type $\E_8$, $A=\PGL(2,31)$ (resp. $A=\SL(2,32)$),
and $K$ a field of characteristic zero.
We view $G$ as a subgroup of $\GL_{248}$ via the adjoint representation.
There is an embedding
$A\to G(K)$ iff $-1$ is a sum of $16$ squares and
$\zeta_{11}+\bar\zeta_{11}\in K$ (for $\SL(2,32)$).

Indeed, by \cite[Theorem~2.27 and Theorem~3.25]{GrRy98}
there exists an embedding $A\to G(\overline k)$ whose character
is defined over $F=\qq$ (resp.\ $F=\qq(\zeta_{11}+\bar\zeta_{11})$).
Using \cite{GrRy98} one can check all conditions of Theorem~\ref{thm2}
(cf. Example \ref{G2.eg}).

It follows by Theorem~\ref{thm2} that there is an embedding $A\to G_0(F)$
for some twisted form $G_0$ of $G$.
Again as in Example~1 one can see that $G_0$ is the unique group such that $G_0\otimes_F\rr$
is compact for all embeddings of $F$ into $\rr$. Finally by Proposition~\ref{pro2}
$G$ and $G_0$ are isomorphic over a field extension $K/F$ iff $-1$ is a sum
of $16$ squares in $K$.

Roughly speaking, we have added the facts about the compact $\E_8$
contained in the proof of Proposition~\ref{pro2} (which uses the existence
of the $u$-invariant) to Serre's appendix \cite[App.~B]{GrRy98}.
\end{eg}

In the same way one can get the following example:

\begin{eg}[type $\A_1$]
Let $G=\PGL_2$, $A=\Alt_4$ (resp. $\Alt_5$), and $K$ a field of characteristic
zero. Then there is an embedding $A\to G(K)$ iff $-1$ is a sum of $2$ squares
and for $\Alt_5$ additionally $\sqrt 5\in K$ (see \cite{Se80}).
\end{eg}

\bibliographystyle{chicago}

\end{document}